\documentclass[a4paper, 12pt]{article}
\usepackage{mathrsfs}
\usepackage{mathrsfs}
\usepackage{amsmath,amssymb}
\usepackage{amsfonts}
\usepackage[T1]{fontenc}
\usepackage[latin9]{inputenc}
\usepackage{amsthm}
\usepackage{graphicx}
\usepackage{amsmath}

\usepackage{amsthm}
\usepackage{array}
\usepackage{cases}
\usepackage{enumerate,enumitem,todonotes}
\makeatletter
\def\th@plain{%
  \itshape 
}
\makeatother

\makeatletter
\renewenvironment{proof}[1][\proofname]{\par
  \pushQED{\qed}%
  \normalfont \topsep6\p@\@plus6\p@\relax
  \trivlist
  \item[\hskip\labelsep
        \bfseries
    #1\@addpunct{.}]\ignorespaces
}{%
  \popQED\endtrivlist\@endpefalse
}
\makeatother

\newtheorem{theorem}{Theorem}[section]

\numberwithin{equation}{section}

\newtheorem{thm}{Theorem}

\newtheorem{lem}[thm]{Lemma}

\newtheorem{question}{Question}

\newtheorem{conj}[thm]{Conjecture}
\newtheorem{pblm}{Problem}
\newtheorem{remark}{Remark}
\newtheorem{clm}{Claim}

\numberwithin{equation}{section}

\usepackage[varg]{txfonts}
\usepackage{graphicx, epsfig, subfigure}
\usepackage{enumerate} 
\usepackage[square, numbers, sort&compress]{natbib} 
\ifx\pdfoutput\undefined
 \usepackage[dvipdfm,%
  pdfstartview=FitH, 
  bookmarks=true,%
  bookmarksnumbered=true, 
  bookmarksopen=true, 
  plainpages=false,%
  pdfpagelabels,%
  colorlinks=true, 
  linkcolor=blue, 
  citecolor=blue,%
  urlcolor=black,
  pdfborder=001]{hyperref}
  \AtBeginDvi{}  
\else
 \usepackage[pdftex,%
  pdfstartview=FitH, 
  bookmarks=true,%
  bookmarksnumbered=true, 
  bookmarksopen=true, 
  plainpages=false,%
  pdfpagelabels,%
  colorlinks=true, 
  linkcolor=blue, 
  citecolor=blue,%
  urlcolor=black,
  pdfborder=001]{hyperref}
\fi

\newcommand{\mad}{\mathrm{mad}}
\numberwithin{equation}{section}

\newcommand{\m}{\mathrm{mad}}
\newcommand{\ad}{\mathrm{ad}}

\usepackage[left]{lineno}

\begin{document}
\title{\LARGE  Odd coloring of 2-boundary planar graphs and beyond
\thanks{Supported by the Natural Science Basic Research Plan in Shaanxi Province of China (No.\,2023-JC-YB-001).}
}
\author{Weichan Liu$^1$ ~~~~ Mengke Qi$^2$ ~~~~ Xin Zhang$^2$\thanks{Corresponding author. Email: xzhang@xidian.edu.cn.}\\
{\small 1. School of Mathematics, Shandong University, Jinan, 250100, China}\\
{\small 2. School of Mathematics and Statistics, Xidian University, Xi'an, 710071, China}}


\maketitle

\begin{abstract}\baselineskip 0.60cm
In this paper, we introduce the notion of 2-boundary planar graphs. A graph is 2-boundary planar if it has an embedding in the plane so that all vertices lie on the boundary of at most two faces and no edges are crossed. A proper coloring of a graph is odd if every non-isolated vertex has some color that appears an odd number of times on its neighborhood. Petru\v{s}evski and \v{S}krekovski conjectured in 2022 that every planar graph admits an odd $5$-coloring. We confirm this conjecture for 2-boundary planar graphs. Moreover, we present several questions regarding 2-boundary planar graphs that are of independent interest.

\vspace{3mm}\noindent \emph{Keywords}: odd coloring; planar graph; 2-boundary planar graph.
\end{abstract}

\baselineskip 0.60cm

\section{Introduction}

Recently, Petru\v{s}evski and \v{S}krekovski \cite{Petruvsevski2021ColoringsWN} introduced the new notion of odd coloring. 
An \textit{odd $k$-coloring} of a graph $G$ is a vertex coloring $\varphi:V(G)\longrightarrow [k]$ such that 
\begin{itemize}
    \item $\varphi(u)\neq \varphi(v)$ for each pair of adjacent vertices $u$ and $v$;
    \item for every non-isolated vertex $v\in V(G)$ there is a color $i\in [k]$ such that the size of $\{u\in N_G(v)~|~\varphi(u)=i\}$ is odd,
\end{itemize}
where $[k]:=\{1,2,\ldots,k\}$.
Let $$\chi_o(G)=\min\{k~|~G~ {\rm has~ an~ odd} ~k{\rm \mbox{-} coloring}\}$$  be the 
\textit{odd chromatic number} of $G$.
Petru\v{s}evski and \v{S}krekovski \cite{Petruvsevski2021ColoringsWN} showed that if $G$ is a planar graph then $\chi_o(G)\leq 9$, and made the following conjecture.

\begin{conj}\label{conj}
$\chi_o(G)\leq 5$ for  every planar graph $G$.
\end{conj}

\noindent It is easy to see that $\chi_o(C_5)=5$, so the bound in this conjecture is sharp if it was correct. 
Recently, Petr and Portier \cite{Petr2022TheOC} improved the bound for the odd chromatic number of planar graphs to 8, which is best known until now.

Towards Conjecture \ref{conj}, Caro, Petru{\v{s}}evski, and {\v{S}}krekovski \cite{zbMATH07585605} showed that every outerplanar graph is odd 5-colorable. Kashima and Zhu \cite{zbMATH07938186} showed that a connected outerplanar graph G is odd 4-colorable if and only if $G$ contains a block which is not a copy of the cycle of length 5. This strengthens the result by Caro, Petru{\v{s}}evski, and {\v{S}}krekovski, and gives a complete characterization of odd 4-colorable outerplanar graphs. Besides, it is natural to consider planar graphs with girth or cycle restriction. 
For each integer $t\leq 7$, let $g_t$ be the minimum integer such that every planar graph with girth at least $g_t$ has an odd $t$-coloring.
Cranston \cite{Cranston2022OddCOO} showed $g_6\leq 6$ and $g_5\leq 7$. Cho, Choi, Kwon, and Park \cite{Cho2022OddCO} 
showed $g_4\leq 11$ and $g_7\leq 5$.
Note that $g_3=+\infty$ as cycles of length not multiple of $3$ are not odd $3$-colorable.
Wang and Yang \cite{zbMATH07785865} showed that if $G$ is a planar graph without $4^-$-cycles adjacent to $7^-$-cycles, then $G$ is odd 6-colorable.

The \textit{average degree} $\ad(G)$ of a graph $G$ is $2|E(G)|/|V(G)|$.
The \textit{maximum average degree} $\m(G)$ of a graph $G$ is the value of $\max\{\ad(H)~|~H\subseteq G\}$.
From the well-known Euler's formula, one can conclude that $\mad(G)<2g/(g-2)$ for every planar graph $G$ with girth at least $g$.
So a natural way to generate the above results on planar graphs with  girth restriction is to investigate the same problem for graphs with bounded maximum average degree. In view of this, Cranston \cite{Cranston2022OddCOO} showed $\chi_o(G)\leq 5$ if $\mad(G)< 20/7$, and $\chi_o(G)\leq 6$ if $\mad(G)< 3$.
Cho, Choi, Kwon, and Park \cite{Cho2022OddCO} proved $\chi_o(G)\leq t$ if $t\geq 7$ and $\mad(G)< 4t/(t-2)$, 
and $\chi_o(G)\leq 4$ if $G$ is an induced $5$-cycle-free graph with $\mad(G)\leq 22/9$. Note that the latter result implies $g_4\leq 11$ that was mentioned before.

Apart from Conjecture \ref{conj},  it is noteworthy that numerous research groups have delved into the study of odd coloring of 1-planar graphs, which are graphs that can be embedded in the plane so that each edge is crossed at most once. 
Cranston, Lafferty, and Song \cite{zbMATH07662546} proved $\chi_o(G)\leq 23$ for every planar graph $G$. This bound was improved to 16 by Niu and Zhang \cite{zbMATH07708542}, and later to 13 by Liu, Wang, and Yu \cite{zbMATH07690011}.

In this paper, we introduce a subclass of planar graphs that generalizes outplanar graphs.
A \textit{2-boundary planar graph} is a graph that has an embedding in the plane so that all vertices lie on the boundary of at most two faces and no edge is crossed, and we call such an embedding \textit{2-boundary plane graph}. Figure \ref{fig:example} gives some examples of 2-boundary plane graphs; note that they are not outerplanar, as we can find a $K_4$-minor in each graph.

\begin{figure}[htp]\label{fig:example}
    \centering
    \includegraphics[width=13cm]{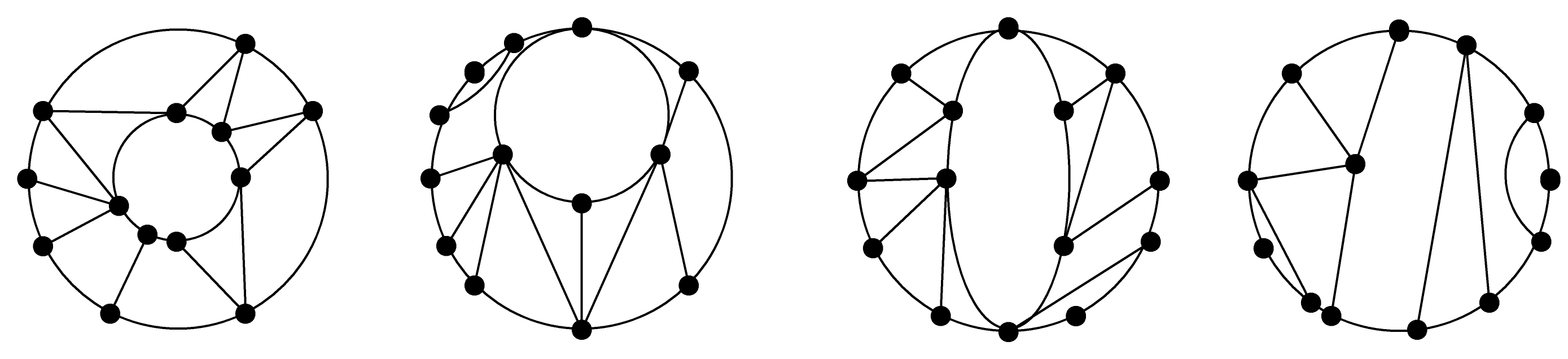}
    \caption{2-boundary planar graphs}
    \label{fig:2bounday}
\end{figure}

In this paper, we verify Conjecture \ref{conj} for 2-boundary planar graphs by the following theorem.
The bound 5 is sharp as $C_5$ is 2-boundary planar and $\chi_o(C_5)=5$.
\begin{thm}\label{thm-2}
Every 2-boundary planar graph has an odd 5-coloring.
\end{thm}

\noindent \textbf{Origination:} In Section \ref{sec:2}, we begin by delineating the structures of 2-boundary planar graphs. Subsequently, in Section \ref{sec:3}, we introduce the concept of odd non-$k$-colorable minimal graphs. We then proceed to prove Theorem \ref{thm-2} by demonstrating that every configuration derived in Section \ref{sec:2} is odd $k$-reducible (see its definition in Section \ref{sec:3}) for all integers $k\geq 5$. In Section \ref{sec:4}, we offer some concluding observations. In particular, we propose several intriguing problems related to 2-boundary planar graphs, which hold independent significance.

\section{Structural properties}\label{sec:2}

For a graph $G$ and a vertex set $S\subseteq V(G)$, $G-S$ denotes the graph obtained from $G$ by removing all vertices of $S$.
If $S=\{v\}$, then we write $G-v$ instead of $G-\{v\}$.
For a graph $G$ and two vertices $u,w\in V(G)$, $G+uw$ denotes the graph obtained from $G$ by connecting $u$ with $w$ with an edge if $uw\not\in E(G)$, or denotes $G$ itself otherwise.

In this section, we focus on 2-boundary planar graphs embedded in the plane without crossings and all vertices lie on the boundaries of exactly two faces\footnote{When there is just one boundary face, the graphs are outerplanar, and Theorem \ref{thm-2} applies based on the findings of Caro, Petru{\v{s}}evski, and {\v{S}}krekovski \cite{zbMATH07585605}. So, this is not the case we are focusing on.}, one of which is unbounded and the other is bounded.
We call the  unbounded face \textit{outer face} and the bounded face \textit{inside face}, denoted by $f_{out}$ and $f_{in}$ respectively. Other faces are called \textit{interior faces}.
By $V(f_{out})$ (resp.\,$V(f_{in})$) we denote the set of vertices on $f_{out}$ (resp.\,$f_{in}$).
Note that it is possible that $V(f_{out})\cap V(f_{in})\neq \emptyset$ (see the last three pictures of Figure \ref{fig:2bounday}). If this occurs, we call vertices in $V(f_{out})\cap V(f_{in})$
\textit{shared vertices}. Let
\[\tilde{\Delta}(G)=\max\bigg\{d_G(u)~\big|~u\in V(G)\setminus \big(V(f_{out})\cap V(f_{in})\big)\bigg\}.\]

For an edge $uv\in E(G)$ such that $\{u,v\}\subseteq V(f_{out})$ or $\{u,v\}\subseteq V(f_{in})$, it is a \textit{chordal edge} if $u$ and $v$ are not consecutive on $f_{out}$ or $f_{in}$, respectively, and \textit{boundary edge} otherwise.
An edge $uv\in E(G)$ is an \textit{interconnected edge} if $u\in V(f_{out})$, $v\in V(f_{in})$, and $\{u,v\}\cap V(f_{out})\cap V(f_{in})=\emptyset$.

A \textit{$k$-vertex} is a vertex $v\in V(G)$ with $d_G(v)=k$. A vertex with odd (resp.\,even)  degree is an \textit{odd vertex} (resp.\,\textit{even vertex}). A neighbor of a vertex is \textit{odd neighbor} if it is an odd vertex.  

A \textit{$k$-fan} with center $v$ is a graph on vertices $v,u_0,\ldots,u_{k-1}$ such that $vu_i$ and $u_iu_{i+1}$ are edges for each possible  $i$, denoted by 
$F[v;u_0\cdots u_{k-1}]$.

\begin{lem}\label{lem-2}
Every $2$-boundary planar graph $G$ contains one of the following configurations unless 
$G$ is isomorphic to $P$ as described in Figure \ref{fig:MP}:
\begin{enumerate}[label={\rm \textbf{(\alph*$\ref{lem-2}$)}}]\setlength{\itemsep}{-3pt}
\item \label{a2} a vertex $v$ of degree $1$;
\item \label{b2} a triangle $uvwu$ with $d_G(v)=2$; 
\item \label{c2} a path $uvw$ with $d_G(v)=2$ and $uw\not\in E(G)$ such that $G-v+uw$ is $2$-boundary planar;
\item \label{d2} a $3$-vertex $v$ with two odd neighbors;

\item \label{e13} 
a triangle $vxyv$ such that $d_G(v)=4$, $d_G(x)=d_G(y)=3$, and all neighbors of $v$ are odd; 
\item \label{h2} a triangle $vxyv$ adjacent to two triangles $uxvu$ and $wvyw$ such that $u\neq w$, $d_G(v)=d_G(x)=d_G(y)=4$, and $d_G(u)=3$; 
\item \label{k21} two adjacent triangles $xuvx$ and $yuvy$ with  $d_G(u)=d_G(v)=3$; 
\item \label{k2} two adjacent triangles $xuvx$ and $yuvy$ with  $d_G(u)=4$ and $d_G(v)=d_G(x)=3$; 
\item \label{g2} two adjacent triangles $xuvx$ and $yuvy$ with  $d_G(u)=d_G(v)=4, d_G(x)=3$ and $d_G(y)$ being odd;

\item \label{i2} a $5$-fan $F[v;u_0u_1u_2u_3u_4]$ such that
$d_G(v)=5$, $d_G(u_1)=d_G(u_3)=4$, and $d_G(u_2)=3$;

\item \label{e11} a $5$-fan $F[v;u_0u_1u_2u_3u_4]$ such that $d_G(v)=6, d_G(u_1)=4$, and $d_G(u_2)=d_G(u_3)=3$;

\item \label{e12} two $5$-fans $F[v;u_0u_1u_2u_3u_4]$ and $F[u_1;u_2vu_0xy]$ such that $d_G(v)=d_G(u_1)=6$, $d_G(u_0)=d_G(x)=d_G(u_2)=d_G(u_3)=3$, and the vertex in $N_G(v)\setminus \{u_0,u_1,u_2,u_3,u_4\}$ is an odd vertex;

\end{enumerate}
\begin{figure}[htp]
    \centering
    \includegraphics[width=4cm]{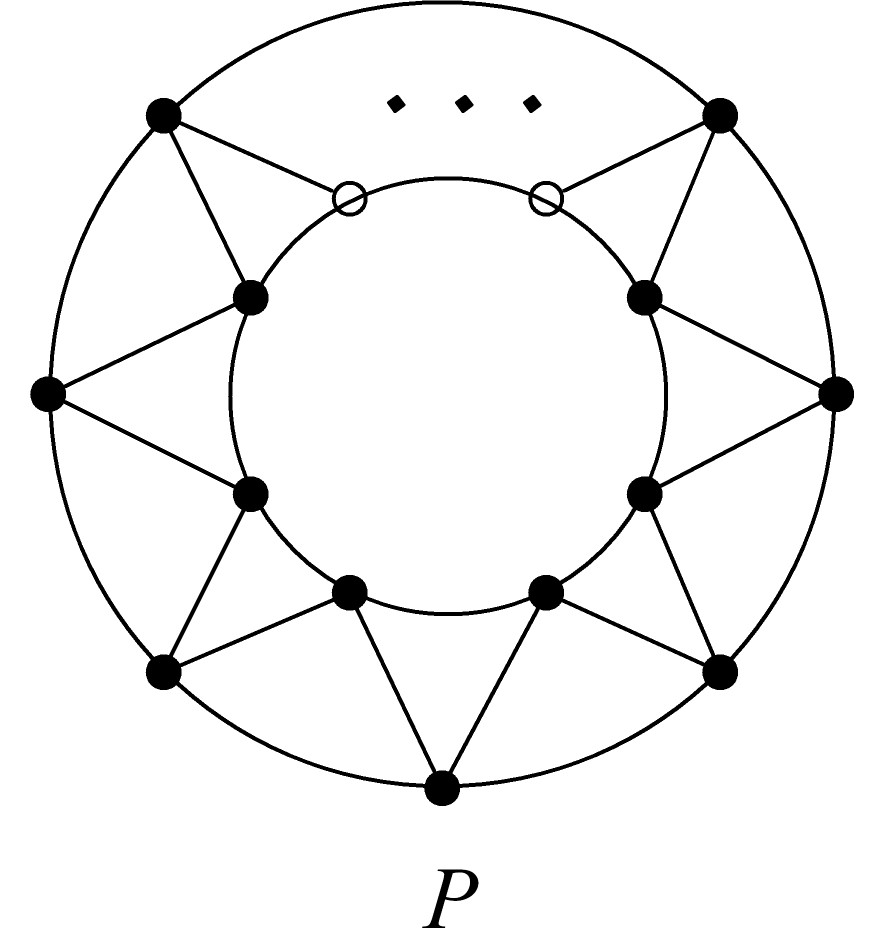}
    \caption{4-regular 2-boundary planar graphs}
    \label{fig:MP}
\end{figure}

\end{lem}

\begin{proof}
Suppose by contradicting that $G$ is a $2$-boundary planar graph such that it contains none of the above configurations and $|G|$ is as small as possible.

If $G$ is an outerplanar graph, then $G$ contains \ref{a2}, \ref{b2}, or \ref{c2} (see \cite{zbMATH00790475,dmtcs:8381}), a contradiciton.
Hence $G$ has an outer face $f_{out}$ and an inside face $f_{in}$.

Clearly, $G$ is connected and $\delta(G)\geq 2$ by the absence of \ref{a2}.

\begin{clm}\label{claim0}
$\delta(G)\geq 3$.
\end{clm}

\proof 
Let $v$ be a $2$-vertex in $G$ such that $N_G(v)=\{u,w\}$. If $v$ is a shared vertex, then $uv$ and $vw$ are boundary edges and $uw\not\in E(G)$. This implies the occurrence of \ref{c2}. If $v$ is not a shared vertex, then
assume w.l.o.g.\,that $v\in  V(f_{out})\setminus V(f_{in})$. 
Now $u,w\in V(f_{out})$. If $uw\in E(G)$, then \ref{b2} occurs. If $uw\not\in E(G)$, then
 $G-v+uw$ is 2-boundary planar and thus \ref{c2} occurs. Hence 
 $\delta(G)\geq 3$. \hfill$\square$

\begin{clm}\label{claim1}
$G$ does not have any chordal edge.
\end{clm}

\proof If $G$ has a chordal edge, then choose one, say $uv$, such that there is no other chordal edge $xy$ such that $u,x,y,v$ lies in this ordering on $f_{out}$ (or $f_{in}$). This implies that there is a $2$-vertex $z$ such that $uz\in E(G)$, contradicting Claim \ref{claim0}. \hfill$\square$


\begin{clm}\label{claim-shared}
If $v$ is a shared vertex then $d_G(v)\leq 4$.
\end{clm}

\proof Assume w.l.o.g.\,$v\in V(f_{out})\cap V(f_{in})$. By Claim \ref{claim1}, $v$ has at most two neighbors in $V(f_{out})$ and at most two neighbors in $V(f_{in})$. Hence $d_G(v)\leq 2+2=4$. \hfill$\square$

\begin{clm}\label{claim2}
If $v\in  V(f_{out})\setminus V(f_{in})$ (resp.\,$v\in  V(f_{in})\setminus V(f_{out})$) and $u_1,u_2,\ldots,u_{d-2}$ $(d\geq 4)$ are all neighbors of $v$ in $V(f_{in})$ (resp.\,$V(f_{out})$) lying in this ordering on $V(f_{in})$ (resp.\,$V(f_{out})$), then $u_iu_{i+1}$ is a boundary edge for each $1\leq i\leq d-3$, and moreover, $d_G(u_i)=3$ for each $2\leq i\leq d-3$.
\end{clm}

\proof
Suppose that there is some vertex $w$ such that $u_i,w,u_{i+1}$ lie in this ordering on $f_{in}$ (resp.\,$f_{out}$) for some $1\leq i\leq d-3$.
Since $G$ does not have any chordal edge by Claim \ref{claim1}, $w$ has at most two neighbors in $V(f_{in})$ (resp.\,$V(f_{out})$). Further, since $d_G(w)\geq 3$ by Claim \ref{claim0}, $w$ has at least one neighbor in $V(f_{out})$ (resp.\,$V(f_{in})$). Since no edge is crossed in $G$, the neighbor of $w$ in $V(f_{out})$ (resp.\,$V(f_{in})$) must be $v$. This is a contradiction as $w$ is not a neighbor of $v$.

So $u_i$ and $u_{i+1}$ are consecutive on $V(f_{in})$ (resp.\,$V(f_{out})$) for each $1\leq i\leq d-3$.
If $u_iu_{i+1}\not\in E(G)$, then $u_ivu_{i+1}$ would be a path on $V(f_{in})$ (resp.\,$V(f_{out})$). This implies $v\in V(f_{in})$ (resp\,$v\in V(f_{out})$), a contradiction. Hence $u_iu_{i+1}\in E(G)$, i.e., $u_iu_{i+1}$ is a boundary edge.
It gives $N_G(u_i)=\{u_{i-1},u_{i+1},v\}$ for each $2\leq i\leq d-3$, implying $d_G(u_i)=3$.\hfill$\square$

\vspace{2mm} Now we divide the main proof into a series of cases.

\vspace{2mm}\textbf{Case 1:} \textit{$\tilde{\Delta}(G)\geq 6$.}\vspace{2mm}

 Choose w.l.o.g.\,a vertex $v\in V(f_{out})\setminus V(f_{in})$ such that $d_G(v)=\tilde{\Delta}(G):=d\geq 6$.
Since $G$ does not have any chordal edge, there are $d-2$ neighbors of $v$, say $u_1,u_2,\ldots,u_{d-2}$, in $V(f_{in})\setminus V(f_{out})$.
We assume that $u_1,u_2,\ldots,u_{d-2}$ lie in this ordering on $f_{in}$. Note that by our assumption none of $u_i$ is a shared vertex.

By Claim \ref{claim2}, $d(u_i)=3$ for each $2\leq i \leq d-3$. 
If $d\geq 7$, then $u_3$ is a $3$-vertex with two odd neighbors $u_2$ and $u_4$,  and thus a configuration \ref{d2} appears. Hence $d=6$.
Similarly, $u_1$ is an even vertex as otherwise $u_2$ is a $3$-vertex with two odd neighbors $u_1$ and $u_3$, implying the appearance of \ref{d2}. 
By Claims \ref{claim0} and \ref{claim-shared}, 
$u_1$ has degree 4 or 6. By symmetry, $u_4$ has degree 4 or 6, too.

If $d_G(u_1)=4$, then there is a vertex $u_0\in V(f_{out})$ such that $u_0u_1\in E(G)$. 
If $vu_0\not\in E(G)$, then by Claim \ref{claim1} there would a $2$-vertex $w$ such that $vw\in E(G)$, contradicting Claim \ref{claim0}.
Hence $vu_0\in E(G)$ and a $5$-fan $F[v;u_0u_1u_2u_3u_4]$ as depicted by \ref{e11} appears.
By symmetry, if $d_G(u_4)=4$ then another copy of \ref{e11} appears in $G$. Hence $d_G(u_1)=d_G(u_4)=6$.

By Claims \ref{claim1} and \ref{claim2}, there are vertices $u_0,x,y\in V(f_{out})$ such that $vu_1xy$ is a path with $d_G(u_0)=d_G(x)=3$.
By symmetry, there is a vertex $z\in V(f_{out})$ such that $vz,zu_4\in E(G)$ and $d_G(z)=3$. Hence two $5$-fans
$F[v;u_0u_1u_2u_3u_4]$ and $F[u_1;u_2vu_0xy]$ as depicted by \ref{e12} appears.

\vspace{2mm}\textbf{Case 2:} \textit{$\tilde{\Delta}(G)=5$.}\vspace{2mm}

 Choose w.l.o.g.\,a vertex $v\in V(f_{out})\setminus V(f_{in})$ such that $d_G(v)=5$.
Since $G$ does not have any chordal edge by Claim \ref{claim1}, there are three neighbors of $v$, say $u_1,u_2,u_{3}$, in $V(f_{in})\setminus V(f_{out})$.
We assume that $u_1,u_2,u_{3}$ lie in this ordering on $V(f_{in})$. By Claim \ref{claim2}, $u_1u_2u_3$ is a path on $f_{in}$ and $d_G(u_2)=3$. 
By the absence of the configuration \ref{d2}, $d_G(u_1)$ and $d_G(u_3)$ are even. Since $u_1,u_3\not\in V(f_{out})$, they are not shared vertices and thus
$d_G(u_1)=d_G(u_3)=4$ as $\tilde{\Delta}(G)=5$.

Since chordal edges are forbidden in $G$, $u_1$ or $u_3$ has a neighbor $u_0$ or $u_4$ on $V(f_{out})\setminus V(f_{in})$
such that $vu_0$ and $vu_4$ are both boundary edges on $f_{out}$
by Claim \ref{claim2}.  Hence a 5-fan $F[v;u_0u_1u_2u_3u_4]$ as depicted by \ref{i2} appears.

\vspace{2mm}\textbf{Case 3:} \textit{$\tilde{\Delta}(G)=4$.}\vspace{2mm}

Choose w.l.o.g.\,a vertex $v\in V(f_{out})\setminus V(f_{in})$ such that $d_G(v)=4$.
Since $G$ does not have any chordal edge by Claim \ref{claim1}, there are two neighbors of $v$, say $u_1,u_2$, in $V(f_{in})\setminus V(f_{out})$. By Claim \ref{claim2}, $u_1u_2$ is a boundary edge on $f_{in}$. 

We claim that it is impossible that $d_G(u_1)= 3$ and $d_G(u_2)= 3$. Otherwise, let $\{v_1\}= N_{G}(u_1)\setminus \{v,u_2\}$ and 
$\{v_2\}= N_{G}(u_2)\setminus \{v,u_1\}$. If $v_1\in V(f_{out})$, then $vv_1$ is a boundary edge by Claim \ref{claim1}, and thus \ref{k2} appears.
Hence $v_1\in V(f_{in})\setminus V(f_{out})$ and by symmetry $v_2\in V(f_{in})\setminus V(f_{out})$. It follows that 
$v_1$ and $v_2$ are not shared vertices. Hence $d_G(v_1)=d_G(v_2)=4$ by Claim \ref{claim0} and by the absence of the configuration \ref{d2}.
By Claim \ref{claim1} , for each $i=1,2$, $v_i$ has two neighbors $x_i$ and $w_i$ on $f_{out}$ such that $x_1w_1vw_2x_2$ is a path of $G$.
It follows that $d_G(w_1)=d_G(w_2)=3$ and thus \ref{e13} appears.

Since $u_1,u_2\not\in V(f_{out})$, they are not shared vertices and thus $d_G(u_1),d_G(u_2)\leq \tilde{\Delta}(G)$.
It follows  $d_G(u_1)=4$ or $d_G(u_2)=4$.

We now assume by symmetry that $d_G(u_2)=4$.
Since chordal edges are forbidden in $G$, there is another vertex besides $v$, say $w_2$, acting as a neighbor of $u_2$ on $V(f_{out})\setminus V(f_{in})$. By Claim \ref{claim2}, $vw_2$ is a  boundary edge on $f_{out}$.

If $d_G(u_1)=3$ and $d_G(w_2)$ is odd, then the configuration \ref{g2} appears.

If $d_G(u_1)=3$ and $d_G(w_2)$ is even, then $d_G(w_2)=4$ as $w_2$ is not a shared vertex. So there is a vertex $u_3$ on $V(f_{in})\setminus V(f_{out})$ such that $w_2u_3\in E(G)$. By Claim \ref{claim2}, $u_2u_3$ is a  boundary edge on $f_{in}$, and thus the configuration \ref{h2} appears.

If $d_G(u_1)=4$, then there is a vertex $w_0$ on $V(f_{out})\setminus V(f_{in})$ such that $w_0u_1\in E(G)$. 
By Claim \ref{claim2}, $w_0v$ is a  boundary edge on $f_{out}$. If $d_G(w_2)=3$, then the configuration \ref{h2} appears.
Hence $d_G(w_2)=4$. So there is a vertex $u_3$ on $V(f_{in})\setminus V(f_{out})$ such that $w_2u_3\in E(G)$. By Claim \ref{claim2}, $u_2u_3$ is a boundary edge on $f_{in}$. If $d_G(u_3)= 3$, then the configuration \ref{h2} appears. Hence $d_G(u_3)= 4$.

Let $w_1:=v$. For each $i\geq 1$,
we repeat the above arguments via letting $w_{i+1}, u_{i+1}, u_{i+2}$ play the roles of $w_i, u_{i}, u_{i+1}$, respectively, and would find two non-shared $4$-vertices $w_{i+2}$ and $u_{i+3}$ such that $w_{i+2}u_{i+2}, w_{i+2}u_{i+3} \in E(G)$ and $w_{i+1}w_{i+2}, u_{i+2}u_{i+3}$ are boundary edges unless \ref{h2} occurs.
Since the graph $G$ is finite, if we forbid the configuration \ref{h2} then 
this iterative process would finally end without finding any new vertex only if $G\cong P$.

\vspace{2mm}\textbf{Case 4:} \textit{$\tilde{\Delta}(G)=3$.}\vspace{2mm}

Choose w.l.o.g.\,a vertex $v\in V(f_{out})\setminus V(f_{in})$ such that $d_G(v)=3$.
Since $G$ does not have any chordal edge by Claim \ref{claim1}, there is one neighbor of $v$, say $u$, in $V(f_{in})\setminus V(f_{out})$.
As $u$ is not a shared vertex, $d_G(u)=3$ by Claim \ref{claim0}.
Let $\{x,y\}\subseteq N_G(u)\setminus \{v\}$. Clearly, $x$ and $y$ are on $f_{in}$.

If $x$ is not a shared vertex, then $3\leq d_{G}(x)\leq \tilde{\Delta}(G)=3$ by Claim \ref{claim0}. 
This implies that the 3-vertex $u$ has two odd neighbors $v$ and $x$ and then \ref{d2} occurs.
Hence $x$ is a shared vertex and so does $y$ by symmetry. It follows that $x,y\in V(f_{out})$ and therefore 
$vx,vy$ are boundary edges on $f_{out}$ by Claim \ref{claim1}. Now, two adjacent triangles $xuvx$ and $yuvy$ as  depicted by \ref{k21} appears.
\end{proof}

\section{Proof of Theorem \ref{thm-2}}\label{sec:3}

For two graphs $H$ and $G$, we write $H< G$ if $|V(H)|+|E(H)|<|V(G)|+|E(G)|$.
Let $\mathcal{G}$ be a class of graphs $G$ such that $H\in \mathcal{G}$ for every subgraph $H$ of $G$.
An \textit{odd non-$k$-colorable minimal graph} (w.r.t.\,$\mathcal{G}$) is a graph $G\in \mathcal{G}$ such that 
\begin{itemize}\setlength{\itemsep}{-3pt}
    \item $\chi_o(G)>k$ and 
    \item $\chi_o(H)\leq k$ for every graph $H\in \mathcal{G}$ with $H<G$.
\end{itemize}
A configuration is \textit{odd $k$-reducible in $\mathcal{G}$} if it cannot occur in an odd non-$k$-colorable minimal graph (w.r.t.\,$\mathcal{G}$).


Let $H,G\in \mathcal{G}$ with $H< G$ and $V(H)\subseteq V(G)$.
For a coloring $\tau$ of $H$ and for a vertex $v\in V(G)$ (note that $v$ may not be colored at this moment), 
if there is a unique color $i$ such that the size of $\{u\in N_H(v)~|~\varphi(u)=i\}$ is odd, then we set 
$\tau_o(v)=i$, and otherwise we set $\tau_o(v)=0$.
Note that $\tau_o(v)\not=0$ implies $d_H(v)$ is odd.

For a set $S$ of vertices we denote 
$\tau(S)=\{\tau(u)~|~u\in S\}$ and $\tau_o(S)=\{\tau_o(u)~|~u\in S\}$. 
We say a color $i$ is odd or even on $S$ if the size of $\{u\in S~|~\tau(u)=i\}$ is odd or even.

Given an odd $k$-coloring $\tau$ of $H$, if we extend it to $G$ then the set of colors that cannot be used by a being colored vertex $v\in V(G)\setminus V(H)$ is therefore
\[F^{\tau}(v)=[k]\cap \bigg(\tau(N_H(v))\cup \tau_o(N_H(v))\bigg).\] Formally, the set $F^{\tau}(v)$ is called the \textit{forbidden set} for $v$, and the set \[A^{\tau}(v)=[k]\setminus F^{\tau}(v)\] 
is called the \textit{available set} for $v$. 

Note that we shall always check the \textit{oddness of $v$}  (i.e.\,whether there is some color that is odd on the neighborhood of $v$) after the extension.

\begin{lem}\label{cl-1}
A vertex $v$ of degree 1 is odd $k$-reducible in $\mathcal{G}$ for every integer $k\geq 3$.
\end{lem}

\begin{proof}
If $G$ is an odd non-$k$-colorable minimal graph containing such a configuration, then 
$G':=G-v$  has an odd $k$-coloring $\tau$. Let $u\in N_G(v)$. 
We extend $\tau$ to $G$ by coloring $v$ with a color in $[k]\setminus F^{\tau}(v)$. 
This is possible since $F^{\tau}(v)=[k]\cap \{\tau(u),\tau_o(u)\}$.
\end{proof}

\begin{lem}\label{cl-2}
A triangle $uvwu$ with $d_G(v)=2$
is odd $k$-reducible in $\mathcal{G}$ for every integer $k\geq 5$.
\end{lem}
\begin{proof}
If $G$ is an odd non-$k$-colorable minimal graph containing such a configuration, then 
$G':=G-v$ has an odd $k$-coloring $\tau$. 
Since $F^{\tau}(v)=[k]\cap \{\tau(u),\tau(w),\tau_o(u),\tau_o(w)\}$, we can color 
$v$ with a color in $[k]\setminus F^{\tau}(v)$. This results in an odd $k$-coloring of $G$
\end{proof}

\begin{lem}\label{cl-3}
A path $uvw$ with $d_G(v)=2$ and $uw$ being a non-edge is odd $k$-reducible in $\mathcal{G}$ for every integer $k\geq 5$ if $G-v+uw\in \mathcal{G}$.
\end{lem}

\begin{proof}
If $G$ is an odd non-$k$-colorable minimal graph containing such a configuration, then  $G':=G-v+uw$ admits an odd $k$-coloring $\tau$ as $G'\in \mathcal{G}$.
Clearly, $\tau(u)\neq \tau(w)$. Assume w.l.o.g. that
$\tau(u)=1$ and $\tau(w)=2$.
Note that $N_{G'}(u)\setminus \{w\} = N_G(u)\setminus \{v\}$ and $N_{G'}(w)\setminus \{u\} =N_G(w)\setminus \{v\}$.
We divide the proofs into three cases by symmetry.

\vspace{2mm}\textbf{Case 1.}  $\tau_o(u)=2$ and $\tau_o(w)=1$

\vspace{2mm} Now every color in $\{3,4,\ldots,k\}$ is even on both $N_G(u)\setminus \{v\}$ and $N_G(w)\setminus \{v\}$.
Hence we color $v$ with 3 to obtain  an odd $k$-coloring of $G$.

\vspace{2mm}\textbf{Case 2.}  $\tau_o(u)=2$ and $\tau_o(w)\not=1$

\vspace{2mm} Similarly, every color in $\{3,4,\ldots,k\}$ is even on $N_G(u)\setminus \{v\}$.
If $\tau_o(w)=0$, then there is a color different from 1, say 3, that is odd on $N_G(w)\setminus \{v\}$.
If $\tau_o(w)\neq 0$, then assume w.l.o.g. that $\tau_o(w)=3$.
In each case we color $v$ with 4 to obtain  an odd $k$-coloring of $G$.

\vspace{2mm}\textbf{Case 3.}  $\tau_o(u)\not=2$ and $\tau_o(w)\not=1$

\vspace{2mm} 
If $\tau_o(u)\not=0$ (assume w.l.o.g.\,$\tau_o(u)=3$), then the colors $2$ and $3$ are odd on $N_G(u)\setminus \{v\}$.
If $\tau_o(u)=0$, then there is a color different from 1 and 2 that is odd on $N_G(u)\setminus \{v\}$.
Hence in each case we find a color $a\neq 1,2$ that is odd on $N_G(u)\setminus \{v\}$.
Similarly, we can find a color $b\neq 1,2$ that is odd on $N_G(w)\setminus \{v\}$.
So we color $v$ with a color in $\{3,4,\ldots,k\}\setminus \{a,b\}$ to obtain an odd 5-coloring of $G$. 
\end{proof}

\begin{lem}\label{cl-5}
A $3$-vertex $v$ with two odd neighbors is odd $k$-reducible in $\mathcal{G}$ for every integer $k\geq 5$.
\end{lem}

\begin{proof}
If $G$ is an odd non-$k$-colorable minimal graph containing a $3$-vertex $v$ with $N_G(v)=\{x,y,z\}$ such that $x$ and $y$ are odd vertices of $G$, then 
$G':=G-v$ admits an odd $k$-coloring $\tau$. 
Now $F^\tau(v)=[k]\cap \{\tau(x),\tau(y),\tau(z),\tau_o(x),\tau_o(y),\tau_o(z)\}$.
Since $x$ and $y$ are even vertices of $G'$, $\tau_o(x)=\tau_o(y)=0$.  
It follows $|A^\tau(v)|\geq k-4\geq 1$.
Hence we color $v$ with a color in $A^\tau(v)$ to obtain an odd 5-coloring of $G$. 
Note that the oddness of $v$ can be always guaranteed as $v$ is an odd vertex of $G$.
\end{proof}

\begin{lem}\label{lem-e13}
A triangle $vxyv$ such that $d_G(v)=4$, $d_G(x)=d_G(y)=3$, and all neighbors of $v$ are odd is odd $k$-reducible in $\mathcal{G}$ for every integer $k\geq 5$.
\end{lem}

\begin{proof}

If $G$ is an odd non-$k$-colorable minimal graph containing such a configuration with $N_G(x)=\{v,y,x'\}$, $N_G(y)=\{v,x,y'\}$, and $N_G(v)=\{x,y,v_1,v_2\}$, then $G':=G-\{x,y\}$ has an odd $k$-coloring $\tau$. 

Now $F^\tau(x)=[k]\cap \{\tau(x'),\tau(v),\tau_o(x')\}$, and $|A^\tau(x)|\geq k-3\geq 2$. We assume w.l.o.g.\,$A^\tau(x)=\{1,2\}$. 

We extend $\tau$ to a $k$-coloring $\phi$ of $G''=G-y$ by coloring $x$ with 1. 
It follows that $A^\phi(y)=[k]\setminus \{\phi(y'),\phi_o(y'),\phi(v),\phi_o(v),\phi(x)\}$. 
If $|A^\phi(y)|\geq 1$, then we extend $\phi$ to an odd $k$-coloring of $G$ by coloring $y$ with a color in $A^\phi(y)$. Hence we assume that $k=5$ and $\{\phi(y'),\phi_o(y'),\phi(v),\phi_o(v)\}=\{2,3,4,5\}$.
Now we adjust the coloring $\phi$ to $\phi'$ by recoloring $x$ with 2. 
A similar argument as above implies $\{\phi'(y'),\phi'_o(y'),\phi'(v),\phi'_o(v)\}=\{1,3,4,5\}$.

Since $\phi(y')=\phi'(y')$ and $\phi(v)=\phi'(v)$, we assume w.l.o.g.\,$\phi(y')=3$ and $\phi(v)=4$. It follows $\phi_o(v)\in \{2,5\}$.

If $\phi_o(v)= 5$, then $\{\phi'(v_1),\phi'(v_2)\}=\{\phi(v_1),\phi(v_2)\}=\{1,5\}$, implying $\phi'_o(v)=0$, a contradiction.

If $\phi_o(v)= 2$, then $\{\phi'(v_1),\phi'(v_2)\}=\{\phi(v_1),\phi(v_2)\}=\{1,2\}$, implying $\phi'_o(v)= 1$ and $\phi_o(y')=\phi'_o(y')=5$.
Now we color $y$ with 4 and recolor $v$ with 5. This results in an odd 5-coloring of $G$.
\end{proof}

\begin{lem}\label{cl-9}
A triangle $vxyv$ adjacent to two triangles $uxvu$ and $wvyw$ such that $u\neq w$, $d_G(v)=d_G(x)=d_G(y)=4$, and $d_G(u)=3$ is odd $k$-reducible in $\mathcal{G}$ for every integer $k\geq 5$.
\end{lem}

\begin{proof}
If $G$ is an odd non-$k$-colorable minimal graph containing such a configuration, then $G':=G-u$ has an odd $k$-coloring $\tau$. Let $x'\in N_G(x)\setminus \{u,v,y\}$, $y'\in N_G(y)\setminus \{x,v,w\}$, and $u'\in N_G(u)\setminus \{x,v\}$.

Now $F^\tau(u)=[k]\cap \{\tau(x),\tau(v),\tau(u'),\tau_o(x),\tau_o(v),\tau_o(u')\}$.
If $k\geq 7$, or $k=6$ and $|F^\tau(u)|\leq 5$, or $k=5$ and $|F^\tau(u)|\leq 4$, then we 
color $u$ with a color in $[k]\setminus F^\tau(u)$ to obtain an odd $k$-coloring of $G$. 
If $k=6$ and $|F^\tau(u)|=6$, then assume w.l.o.g.\,$\tau_o(x)=1$, $\tau_o(v)=2$, $\tau(x)=3$, and $\tau(v)=4$. 
However $\tau_o(x)=1$ implies $\tau(y)=1$ but 
$\tau_o(v)=2$ implies $\tau(y)=2$, a contradiction. Hence $k=5$ and $\{\tau(x),\tau(v),\tau(u'),\tau_o(u'),\tau_o(x),\tau_o(v)\}=[5]$.
It follows that $\tau_o(x)\neq 0$ or $\tau_o(v)\neq 0$.

\vspace{2mm}\textbf{Case 1.}  $\tau_o(v)\neq 0$.

\vspace{2mm}
Assume w.l.o.g.\,$\tau(v)=1$ and $\tau_o(v)=2$. It follows $\tau(y)=2$ and $\tau(w)=\tau(x)$. Assume $\tau(x)=3$. 

\vspace{2mm}\textbf{Subcase 1.1}  $\tau_o(x)\neq 0$.

\vspace{2mm}
This situation implies $\{\tau_o(x),\tau(x')\}=\{1,2\}$. 
Since $$[5]=\{\tau(x),\tau(v),\tau(u'),\tau_o(x),\tau_o(v),\tau_o(u')\}=\{1,2,3,\tau(u'),\tau_o(u')\},$$ $\{\tau(u'),\tau_o(u')\}=\{4,5\}$. 
Recolor $x$ with a color in $\{4,5\}\setminus \{\tau_o(x')\}$
and then color $u$ with 3.
One can check that the resulting coloring is an odd $5$-coloring of
$G$. 

\vspace{2mm}\textbf{Subcase 1.2}  $\tau_o(x)=0$.

\vspace{2mm}
This situation implies $|\{\tau(x'),\tau(y),\tau(v)\}|=3$. So we assume w.l.o.g.\,$\tau(x')=4$.
Since $$[5]=\{\tau(x),\tau(v),\tau(u'),\tau_o(x),\tau_o(v),\tau_o(u')\}=\{1,2,3,\tau(u'),\tau_o(u')\},$$
$\{\tau(u'),\tau_o(u')\}=\{4,5\}$.
If $\tau_o(x')\neq 5$, then recolor $x$ with 5 and color $u$ with 3. 
If $\tau_o(x')=5$, then recolor $x$ with 1 and $v$ with a color in $\{4,5\}\setminus \{\tau_o(w)\}$, and color $u$ with 3.
In each case we obtain an odd $5$-coloring of $G$.

\vspace{2mm}\textbf{Case 2.}  $\tau_o(v)= 0$.

\vspace{2mm}

In this situation we have $[5]=\{\tau(x),\tau(v),\tau(u'),\tau_o(u'),\tau_o(x)\}$. 
So we assume w.l.o.g.\,that  
$\tau(x)=1$, $\tau_o(x)=2$, $\tau(v)=3$, $\tau(u')=4$, and $\tau_o(u')=5$.
Now $\tau_o(x)=2$ forces $\tau(y)=2$ and $\tau(x')=3$. 




Since $\tau_o(v)= 0$,  $|\{\tau(w),\tau(x),\tau(y)\}|=3$. So $\tau(w)\in \{4,5\}$. Let $\tau(w)=a$ and $b\in \{4,5\}\setminus \{a\}$.

If $\tau_o(x')=a$, then recolor $x$ with $b$ and color $u$ with 1. 
If $\tau_o(x')\not=a$ and $\tau(y')\not=3$, then recolor $x$ with $a$ and color $u$ with 1.
If $\tau_o(x')\not=a$ and $\tau(y')=3$, then recolor $x$ with $a$ and $v$ with a color in $\{1,b\}\setminus \{\tau_o(w)\}$, and color $u$ with 3.
In each case we obtain an odd $5$-coloring of $G$.
\end{proof}

\begin{lem}\label{cl-4}
Two adjacent triangles $xuvx$ and $yuvy$ with $d_G(u)=d_G(v)=3$ 
is odd $k$-reducible in $\mathcal{G}$ for every integer $k\geq 5$.
\end{lem}
\begin{proof}
If $G$ is an odd non-$k$-colorable minimal graph containing such a configuration, then 
$G':=G-v$ admits an odd $k$-coloring $\tau$. 
Now $A^\tau(v)=[k]\setminus \{\tau(x),\tau(y),\tau(u),\tau_o(x),\tau_o(y)\}$. 
If $|A^\tau(v)|\geq 1$, then it is possible to color $v$ with a color in $A^\tau(v)$ to complete an odd $k$-coloring of $G$. So we assume below that $k=5$
and assume further w.l.o.g.\,$\tau(x)=1,\tau(y)=2,\tau(u)=3,\tau_o(x)=4,\tau_o(y)=5$. 
Now we recolor $u$ with 4 and color $v$ with 5. The resulting coloring of $G$ is an odd 5-coloring of $G$ as the color 3 is odd on the neighborhood of both $x$ and $y$.
\end{proof}

\begin{lem}\label{lem-k2}
Two adjacent triangles $yuvy$ and $xuvx$ with  $d_G(u)=4$ and $d_G(v)=d_G(x)=3$
is odd $k$-reducible in $\mathcal{G}$ for every integer $k\geq 5$.
\end{lem}

\begin{proof}
If $G$ is an odd non-$k$-colorable minimal graph containing such a configuration, then $G':=G-v$ admits an odd $k$-coloring $\tau$. 

Now $A^\tau(v)=[k]\setminus
\{\tau(x),\tau(y),\tau(u),\tau_o(y),\tau_o(u)\}$ as $d_{G'}(x)=2$ and $\tau_o(x)=0$. If $|A^\tau(v)|\geq 1$, then 
we color $v$ with a color in $A^\tau(v)$ to obtain an odd $k$-coloring of $G$. So we assume below that $k=5$
and assume further w.l.o.g.\,$\tau(x)=1,\tau(y)=2,\tau(u)=3,\tau_o(y)=4,\tau_o(u)=5$. 
However, since $x$ and $y$ are neighbors of $u$, $5=\tau_o(u)\in \{\tau(x),\tau(y)\}=\{1,2\}$, a contradiction. 
\end{proof}

\begin{lem}\label{cl-8}
Two adjacent triangles $xuvx$ and $yuvy$ with  $d_G(u)=d_G(v)=4, d_G(x)=3$ and $d_G(y)$ being odd is odd $k$-reducible in $\mathcal{G}$ for every integer $k\geq 5$.
\end{lem}

\begin{proof}
If $G$ is an odd non-$k$-colorable minimal graph containing such a configuration, then $G':=G-x$ admits an odd $k$-coloring $\tau$. Let $x'\in N_G(x)\setminus \{u,v\}$, $u'\in N_G(u)\setminus \{x,y,v\}$, and $v'\in N_G(v)\setminus \{x,y,u\}$.

If we are able to color $x$ with a color $a\in [k]\setminus \{\tau(u),\tau(v),\tau(y),\tau(x'),\tau_o(x')\}$, then we obtain an odd $k$-coloring of $G$.
Note that the oddnesses on $u$ and $v$ can be guaranteed as $|\{a,\tau(v),\tau(y)\}|=3$ and $|\{a,\tau(u),\tau(y)\}|=3$.

So we assume below that $k=5$ and 
assume further w.l.o.g.\,$\tau(u)=1,\tau(v)=2,\tau(y)=3,\tau(x')=4,\tau_o(x')=5$. 
If $\tau(u')\neq 2$ and $\tau(v')\neq 1$, then we complete an odd $5$-coloring of $G$ by coloring $x$ with 3.
Otherwise, we assume w.l.o.g.\,$\tau(u')= 2$.

Now erase the color on $u$ and denote the resulting odd $k$-coloring of
$G''=G'-u=G-\{u,x\}$ by $\phi$. Since $d_{G''}(y)$ and $d_{G''}(v)$ are even, $\phi_o(y)=\phi_o(v)=0$.
Hence $F^\phi(u)=[5]\cap \{\phi(y),\phi(v),\phi(u'),\phi_o(u')\}=[5]\cap \{2,3,\phi_o(u')\}$ and $A^\phi(u)=\{1,4,5\}\setminus \{\phi_o(u')\}$.
We color $u$ with a color $a\in A^\phi(u)\setminus \{1\}$ and $x$ with 1 to finish a 
 an odd 5-coloring of $G$.
\end{proof}


\begin{lem}\label{cl-10}
A $5$-fan $F[v;u_0u_1u_2u_3u_4]$ such that
$d_G(v)=5$, $d_G(u_1)=d_G(u_3)=4$, and $d_G(u_2)=3$ is odd $k$-reducible in $\mathcal{G}$ for every integer $k\geq 5$.
\end{lem}

\begin{proof}
If $G$ is an odd non-$k$-colorable minimal graph containing such a configuration, then $G':=G-u_2$ has an odd $k$-coloring $\tau$. Let $x\in N_G(u_1)\setminus \{u_2,v,u_0\}$ and $y\in N_G(u_3)\setminus \{u_2,v,u_4\}$.
Note that $\tau_o(v)=0$ as $d_{G'}(v)=4$. 

Now $F^\tau(u_2)=[k]\cap \{\tau(u_1),\tau(v),\tau(u_3),\tau_o(u_1),\tau_o(u_3)\}$. 
If $k\geq 6$, then we color $u_2$ with a color in $[k]\setminus F^\tau(u_2)$ to complete an odd $k$-coloring of $G$. So we assume $F^\tau(u_2)=[k]=[5]$ and assume further w.l.o.g.\,$\tau(u_1)=1,\tau(v)=2,\tau(u_3)=3,\tau_o(u_1)=4$, and $\tau_o(u_3)=5$. 
It follows $ \tau(x)=\tau(y)=2, \tau(u_0)=4,$ and $\tau(u_4)=5$.

If $\{\tau_o(x), \tau_o(u_0)\}\neq \{3,5\}$, then recolor $u_1$ with a color in $\{3,5\}\setminus \{\tau_o(x), \tau_o(u_0)\}$  and color $u_2$ with 1. 
If $\{\tau_o(y), \tau_o(u_4)\}\neq \{1,4\}$, then recolor $u_3$ with a color in $\{1,4\}\setminus \{\tau_o(y), \tau_o(u_4)\}$  and color $u_2$ with 3.
In each case one can easily check that the resulting coloring is an odd $5$-coloring of $G$.

If $\{\tau_o(x), \tau_o(u_0)\}= \{3,5\}$ and $\{\tau_o(y), \tau_o(u_4)\}= \{1,4\}$, 
then the color 2 is even (appearing at least twice) on $N_{G'}(u_0)$ and $N_{G'}(u_4)$. We recolor $u_1$
with $\tau_o(u_0)$ and change the color of $v$ from $2$ to $1$. 
Note that $\tau_o(u_0)\neq \tau_o(x)$ and the color 2 is now odd on $N_{G}(u_0)$ and $N_{G}(u_4)$ no matter which color is assigned to $u_2$.
Hence we can color $u_2$ with 2 to obtain an 
odd $5$-coloring of $G$. 
\end{proof}

\begin{lem}\label{lem-e11}
A $5$-fan $F[v;u_0u_1u_2u_3u_4]$ such that $d_G(v)=6, d_G(u_1)=4$, and $d_G(u_2)=d_G(u_3)=3$
is odd $k$-reducible in $\mathcal{G}$ for every integer $k\geq 5$.
\end{lem}

\begin{proof}

If $G$ is an odd non-$k$-colorable minimal graph containing such a configuration, then $G':=G-\{u_2,u_3\}$ has an odd $k$-coloring $\tau$. Let $v_0 \in N_G(v)\setminus \{u_0,u_1,u_2,u_3,u_4\}$. 

Now $F^\tau(u_2)=[k]\cap \{\tau(u_1),\tau(v),\tau_o(u_1)\}$, and $|A^\tau(u_2)|\geq k-3\geq 2$. We assume w.l.o.g.\,$A^\tau(u_2)=\{1,2\}$. 

We extend $\tau$ to a $k$-coloring $\phi$ of $G''=G-u_3$ by coloring $u_2$ with 1. 
It follows that $A^\phi(u_3)=[k]\setminus \{\phi(u_4),\phi_o(u_4),\phi(v),\phi_o(v),\phi(u_2)\}$. 
If $|A^\phi(u_3)|\geq 1$, then we extend $\phi$ to an odd $k$-coloring of $G$ by coloring $u_3$ with a color in $A^\phi(u_3)$. Hence we assume that $k=5$ and $\{\phi(u_4),\phi_o(u_4),\phi(v),\phi_o(v)\}=\{2,3,4,5\}$.
We now adjust the coloring $\phi$ to $\phi'$ by recoloring $u_2$ with 2. A similar argument as above would imply 
$\{\phi'(u_4),\phi'_o(u_4),\phi'(v),\phi'_o(v)\}=\{1,3,4,5\}$.



Since $\phi(u_4)=\phi'(u_4), \phi_o(u_4)=\phi'_o(u_4)$, and $\phi(v)=\phi'(v)$, we conclude that
$\{\phi(u_4),\phi_o(u_4),\phi(v)\}= \{3,4,5\}$ and thus
$\phi_o(v)=2$.
Assume w.l.o.g.\,$\phi(v)=5$, $\phi(u_4)=4$, and $\phi_o(u_4)=3$.

Since $\phi(u_1)=\tau(u_1)\not\in \{1,2\}$ and $\phi_o(v)=2$, we have $\{\phi(u_0),\phi(v_0)\}=\{1,2\}$ and $\phi(u_1)=\phi(u_4)=4$. 
We now color $u_3$ with 1 and recolor $u_2$ with 3. One can check that the resulting coloring is an odd 5-coloring of $G$. Note that the oddness of $v$ is satisfied as $u_2$ is the unique neighbor of $v$ that is colored with 3, and the oddness of $u_1$ is also guaranteed as $v,u_0,u_2$ are colored different.
\end{proof}

\begin{lem}\label{lem-e12}
Two $5$-fans $F[v;u_0u_1u_2u_3u_4]$ and $F[u_1;u_2vu_0xy]$ such that $d_G(v)=d_G(u_1)=6$, $d_G(u_0)=d_G(x)=d_G(u_2)=d_G(u_3)=3$, and the vertex in $N_G(v)\setminus \{u_0,u_1,u_2,u_3,u_4\}$ is an odd vertex
is odd $k$-reducible in $\mathcal{G}$ for every integer $k\geq 5$.
\end{lem}

\begin{proof}

If $G$ is an odd non-$k$-colorable minimal graph containing such a configuration, then $G':=G-\{u_2,u_3\}$ has an odd $k$-coloring $\tau$. Let $N_G(v)=\{w_1,u_0,u_1,u_2,u_3,u_4\}$ and $N_G(u_1)= \{w_2,u_2,v,u_0,x,y\}$.

Now $F^\tau(u_2)=[k]\cap \{\tau(u_1),\tau(v),\tau_o(u_1)\}$, and $|A^\tau(u_2)|\geq k-3\geq 2$. We assume w.l.o.g.\,$A^\tau(u_2)=\{1,2\}$. 
We extend $\tau$ to a $k$-coloring $\phi$ of $G''=G-u_3$ by coloring $u_2$ with 1. 
It follows that $A^\phi(u_3)=[k]\setminus \{\phi(u_4),\phi_o(u_4),\phi(v),\phi_o(v),\phi(u_2)\}$. 
If $|A^\phi(u_3)|\geq 1$, then we extend $\phi$ to an odd $k$-coloring of $G$ by coloring $u_3$ with a color in $A^\phi(u_3)$. Hence we assume that $k=5$ and $\{\phi(u_4),\phi_o(u_4),\phi(v),\phi_o(v)\}=\{2,3,4,5\}$.
We adjust the coloring $\phi$ to $\phi'$ by recoloring $u_2$ with 2. 
A similar argument as above implies
$\{\phi'(u_4),\phi'_o(u_4),\phi'(v),\phi'_o(v)\}=\{1,3,4,5\}$.

Since $\phi(u_4)=\phi'(u_4), \phi_o(u_4)=\phi'_o(u_4)$, and $\phi(v)=\phi'(v)$, we conclude that 
$\{\phi(u_4),\phi_o(u_4),\phi(v)\}= \{3,4,5\}$ and therefore
$\phi_o(v)=2$ and $\phi'_o(v)=1$.  
Assume w.l.o.g.\,$\phi(v)=5$, $\phi(u_4)=3$, and $\phi_o(u_4)=4$. 
Since $\phi(u_1)=\tau(u_1)\not\in \{1,2\}$ and $\phi_o(v)=2$, $\phi(u_1)=\phi(u_4)=3$ and  
$\{\phi(u_0),\phi(w_1)\}= \{1,2\}$. Let $\phi(u_0)=a$ and $b\in \{1,2\}\setminus \{a\}$.

We now, starting from $\phi$ again, recolor $v$ with 4 and $u_2$ with 5, and then color $u_3$ with 2. Denote the resulting coloring by $\theta$.
Note that the color 5 is even (appearing at least twice) on $N_{G}(u_4)\setminus \{u_3\}$ under $\phi$, and thus it 
is odd  on $N_{G}(u_4)$ under $\theta$, so the oddness of $u_4$ is satisfied.

If $\{\theta(x),\theta(y),\theta(w_2)\}\neq \{a,4,5\}$, then the oddness 
of $u_1$ is guaranteed as $\{\theta(v),\theta(u_0),\theta(u_2)\}= \{a,4,5\}$. 
If $\{\theta(x),\theta(y),\theta(w_2)\}= \{a,4,5\}$, then recolor $u_0$ with $b$, and thus 
the oddness  of $u_1$ is also satisfied as $u_0$ is the unique neighbor of $u_1$ that is colored with $b$.

In each case, the oddness of $v$ is satisfied as $u_2$ is the unique neighbor of $v$ that is colored with 5, and the oddnesses of  $x,u_0,u_2,u_3$ and $w_1$ are preserved definitely as they are odd vertices.
Hence $\theta$ is odd, as desired.
\end{proof}

Now we are ready to prove Theorem \ref{thm-2}. Suppose by contradiction that there is 
an odd non-$5$-colorable minimal 2-boundary planar graph $G$.

If $G\cong P$, then assume $V(f_{out})=\{u_1,u_2,\dots,u_n\}$, $V(f_{in})=\{v_1,v_2,\dots,v_n\}$, $N_G(v_1)=\{u_n, u_1, v_n, v_2\}$, $N_G(v_n)=\{u_{n-1}, u_n, v_{n-1}, v_1\}$, and $N_G(v_i)=\{u_{i-1}, u_i, v_{i-1}, v_{i+1}\}$ for each $i=2,3,\dots,n-1$. 
If $n$ is even, we construct an odd $4$-coloring $\tau$ of $G$ by letting $\tau(u_i)=1, \tau(u_{i+1})=2, \tau(v_i)=3$, and $\tau(v_{i+1})=4$ for $i=1,3,\dots,n-1$. 
If $n$ is odd, we construct an odd $5$-coloring $\tau$ of $G$ by letting 
$\tau(u_i)=4, \tau(u_{i+1})=5$ for $i=1,3,\ldots,n-4,n-1$, $\tau(u_{n-2})=3$, 
$\tau(v_i)=1$, $\tau(v_{i+1})=2$ for $i=1,3,\dots,n-2$, and $\tau(v_n)=3$.
Hence $\chi_o(G)\leq 5$, a contradiction.

It follows $G\not\cong P$ and thus 
$G$ contains at least one configuration among \ref{a2}, \ref{b2}, \ref{c2}, \ref{d2}, \ref{e13}, \ref{h2}, \ref{k21}, \ref{k2}, \ref{g2}, \ref{i2}, \ref{e11}, and \ref{e12} by Lemma \ref{lem-2}. However, \ref{a2}, \ref{b2}, \ref{c2}, \ref{d2}, \ref{e13}, \ref{h2}, \ref{k21}, \ref{k2}, \ref{g2}, \ref{i2}, \ref{e11}, and \ref{e12} are respectively odd 5-reducible by Lemma \ref{cl-1}, \ref{cl-2}, \ref{cl-3}, \ref{cl-5}, \ref{lem-e13}, \ref{cl-9}, \ref{cl-4}, \ref{lem-k2}, \ref{cl-8}, \ref{cl-10}, \ref{lem-e11}, and \ref{lem-e12}, a contradiction. \hfill $\blacksquare$

\vspace{3mm} 

We end this section with the following open problem:

\begin{pblm} \label{problem-1}
Characterize all $2$-boundary planar graphs that have odd chromatic number exactly $5$.
\end{pblm}

\noindent 
Towards Problem \ref{problem-1}, we propose the following conjecture:

\begin{conj}
A connected $2$-boundary planar graph $G$ has odd chromatic number exactly $5$ if and only if 
every block of $G$ is isomorphic to $C_5$.
\end{conj}

\section{Concluding Remarks} \label{sec:4}

This paper introduces the concept of 2-boundary planar graphs. While our primary objective is to verify the conjecture of Petru\v{s}evski and \v{S}krekovski for this particular class of graphs, it is worth noting that this new graph class possesses its own independent interest.

First, we focus on a Tur\'an-type problem, which ask for the maximum number of edges a 2-boundary planar graph can have. 

\begin{theorem} \label{edge}
Every 2-boundary planar graph on $n$ vertices has at most $2n$ edges, and this bound is tight.
\end{theorem}

\begin{proof}
Given a 2-boundary planar graph $G$, we iteratively perform the following operations until no further operations can be applied:

\begin{itemize} 
    \vspace{-.5em}\item delete a vertex of degree 1;
    \vspace{-.5em}\item delete a vertex of degree 2;
    \vspace{-.5em}\item for each vertex $u$ of degree 3, 
    \begin{itemize}
       \vspace{-.5em} \item if $u$ is not incident with any chordal edge, then $u$ is incident with two boundary edges $uv$ and $uw$, and an interconnected edge $ux$. In this case, we delete edges $uv,uw,ux$ and creat a new edge $vw$; 
       \vspace{-.5em} \item if $u$ is incident with a chordal edge $uz$, then $u$ is incident with two boundary edges $ux$ and $uy$. In this case, we delete edges $ux,uy,uz$ and create a new edge $xy$.
    \end{itemize}
\end{itemize}

Let $H$ be the resulting graph, which has $n'$ vertices. Clearly, $H$ is 2-boundary planar and $\delta(H)\geq 4$, and the above operation implies $e(H)\geq e(G)-2(n-n')$. WIthout loss of generality, assume that $H$ is connected.

Assume that $u$ is a vertex of degree at least 5. 
If there exists a chordal edge $uv$, then there must be a vertex $w$ that lies on the path from $u$ to $v$ along either $f_{out}$ or $f_{in}$, and $w$ has a degree of 2, a contradiction. Hence we assume that $u$ is not incident with any chord, and thus $u\not\in V(f_{out})\cap V(f_{in})$. 
In this case, $u$ is incident with two boundary edges and at least three interconnected edges. Without loss of generality, let $ux,uy,uz$ be three interconnected edges such that $z$ lies on the path from $x$ to $y$ along either $f_{out}$ or $f_{in}$. Now, if $z$ is incident with a chordal edge, then similarly we would see a vertex of degree 2, and if not, then $z$ itself has degree of 3. Both are contradictions. Therefore, 
$\Delta(H)\leq 4$, and thus $H$ is 4-regular. It follows $e(H)=2n'$, and thus $e(G) \leq e(H)+2(n-n')=2n$, as desired.

For the tightness of this bound $2n$, see Figure \ref{fig:MP}. 
\end{proof}

According to Lemma \ref{lem-2}, every 2-boundary planar graph, except for the 4-regular 2-boundary planar graph $P$ illustrated in Figure \ref{fig:MP}, has a vertex with a degree of at most 3. A natural question is whether we can lower the bound $2n$ in Theorem \ref{edge} if we prohibit the graph $P$. However, this is not the case, as we have another types of extremal graphs which have minimum degree 3 or 2, see Figures \ref{fig:MQ} and \ref{fig:MR}. 
\begin{figure}[htp]
    \centering
    \includegraphics[width=4cm]{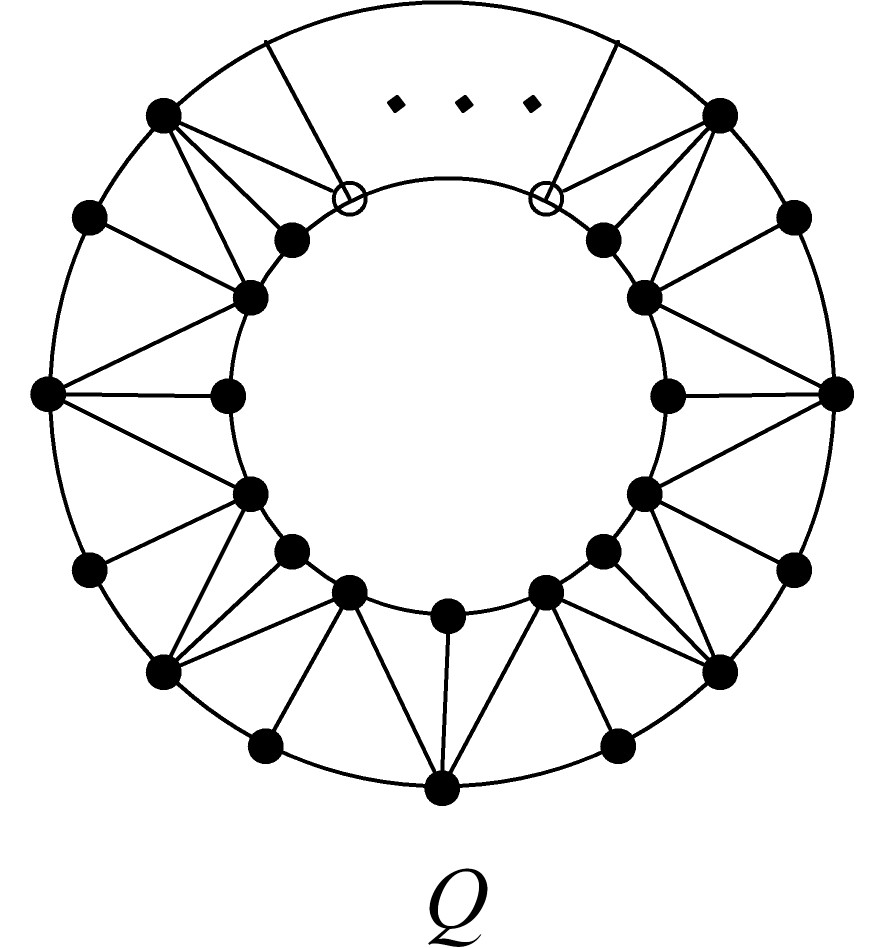}
    \caption{A 2-boundary planar graph with minimum degree 3 and $4n$ edges}
    \label{fig:MQ}
\end{figure}
\begin{figure}[htp]
    \centering
    \includegraphics[width=4cm]{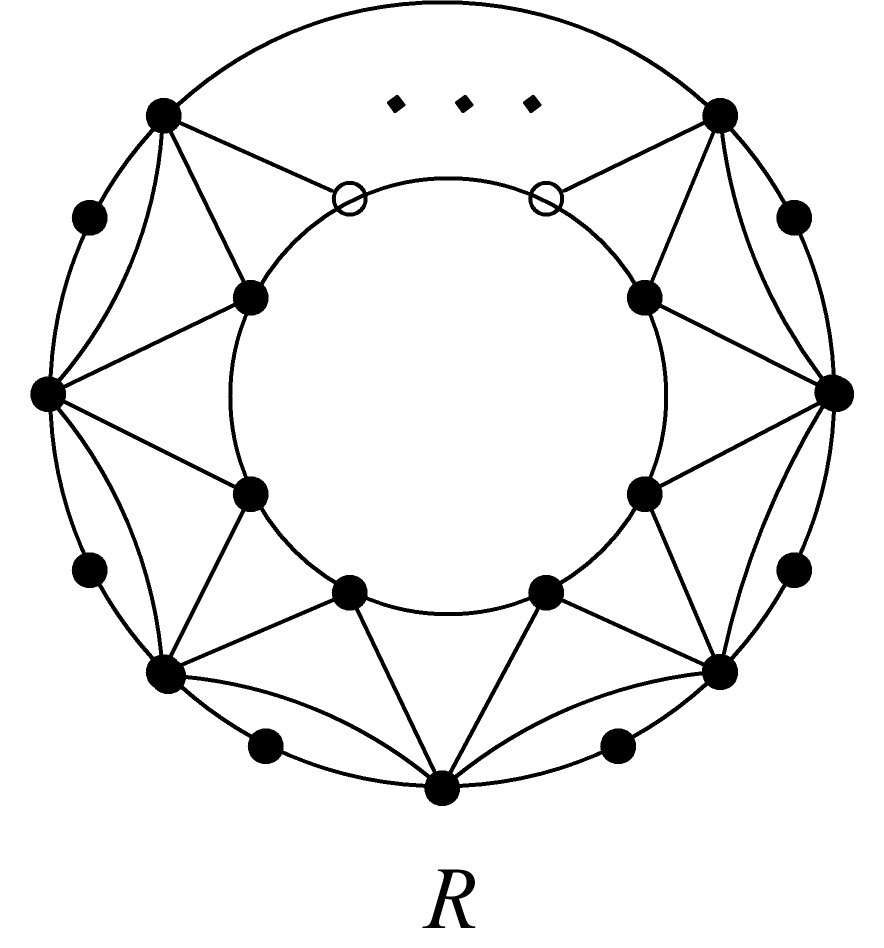}
    \caption{A 2-boundary planar graph with minimum degree 2 and $4n$ edges}
    \label{fig:MR}
\end{figure}

A 2-boundary planar graph $G$ is \textit{maximal} if the addition of any edge would result in the graph no longer being 2-boundary planar. 
Given an $n$-vertex maximal 2-boundary planar graph $G$, if $V(f_{out})\cap V(f_{in})=\emptyset$, then it has exactly $2n$ edges. Actually,
we can add $|f_{out}|-3$ edges outside $f_{out}$ and $|f_{in}|-3$ edges inside $f_{in}$ to make $G$ being a maximal planar graph, which has exactly $3n-6$ edges. Therefore, the number of edges in $G$ is exactly $(3n-6)-(|f_{out}|-3)-(|f_{in}|-3)=2n$, as $|f_{out}|+|f_{in}|=n$. However, if $|V(f_{out})\cap V(f_{in})|=\ell\geq 1$, then by a similar calculation we conclude that $G$ has at most $2n-\ell$ edges. This indicates that for a 2-boundary planar graph, the maximal property does not imply the maximum property, which is a phenomenon distinct from that of planar and outerplanar graphs. This motivates a natural question:

\begin{question}
How many edges must a maximal $2$-boundary planar graph with $n$ vertices have at minimum?
\end{question}

Theorem \ref{edge} finds an interesting application in edge coloring. The celebrated Vizing's theorem asserts that for any graph $G$ the chromatic index $\chi'(G)$ is either $\Delta(G)$ or $\Delta(G)+1$.
Graphs $G$ are classified as \textit{class 1} if $\chi'(G) = \Delta(G)$, and as \textit{class 2} if $\chi'(G) = \Delta(G) + 1$.
A graph $G$ is \textit{critical} if it is connected, belongs to class 2, and satisfies $\chi'(G - e) < \chi'(G)$ for every edge $e$ in $G$. When a graph $G$ is critical and has a maximum degree of $\Delta$, we refer to it as \textit{$\Delta$-critical}. Yap \cite{zbMATH03762083} showed that if $G$ be a $\Delta$-critical graph with $\Delta\geq 5$, then $G$ has at least $2|G|+1$ edges. Now it is sufficient to prove the following.

\begin{theorem} \label{ec}
If $G$ is a 2-boundary planar graph and $\Delta(G)\geq 5$, then
$\chi'(G)=\Delta(G)$.
\end{theorem}

\begin{proof}
If not, then $G$ is of class 2 and thus has a $\Delta(G)$-critical subgraph $G'$. The result of Yap implies that $G'$ has at least $2|G'|+1$ edges. However, $G'$ is a 2-boundary planar graph and thus has at most $2|G'|$ edges by Theorem \ref{edge}. This contradiction ends the proof.
\end{proof}

The maximum degree bound $5$ in Theorem \ref{ec} is tight because there exists 2-boundary planar graphs with maximum degree 4 and chromatic index 5, see Figure \ref{fig:class2}.
\begin{figure}[htp]
    \centering
    \includegraphics[width=4cm]{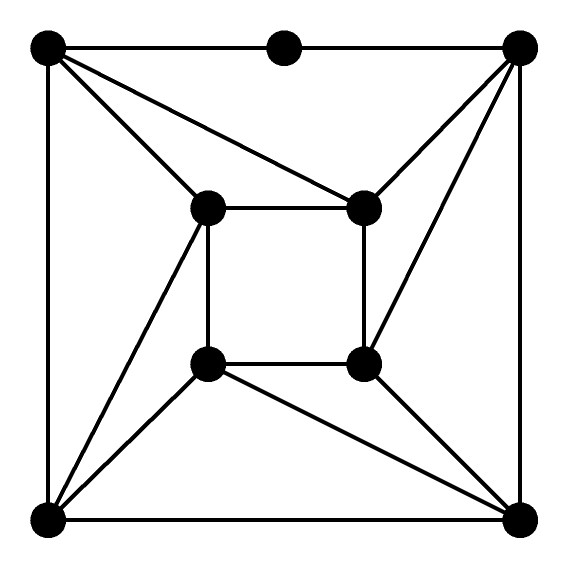}
    \caption{A class 2 2-boundary planar graph with maximum degree 4}
    \label{fig:class2}
\end{figure}

The \textit{dual} $G^*$ of a plane graph $G$ is a graph formed by placing a vertex in each face of $G$ and connecting vertices corresponding to adjacent faces with edges. The \textit{weak dual} of a plane graph $G$ is constructed by removing specific vertices from its dual graph $G^*$. For an outerplane graph $G$, this involves deleting the vertex in $G^*$ corresponding to the outer face (unbounded face) of $G$. In the case of a 2-boundary plane graph $G$, the weak dual is formed by removing both vertices in $G^*$ that represent the outer face $f_{out}$ and inside face $f_{in}$ of $G$. It is well-known that the weak dual of any outerplane graph is a forest,  while the weak dual of a biconnected outerplane graph is a tree.  In this following, we show that weak dual of any 2-boundary plane graph is a forest or a unicyclic graph, a graph possessing exactly one cycle.

\begin{theorem} \label{unicyclic}
The weak dual of a 2-boundary plane graph is a forest or a unicyclic graph.
\end{theorem}

\begin{proof}
We divide our proof into two cases as follows:

\vspace{2mm}\textbf{Case 1:} $V(f_{out})\cap V(f_{in})\not=\emptyset$.\vspace{2mm}

Let $v\in V(f_{out})\cap V(f_{in})$. We perform a vertex splitting operation on graph $G$ at vertex $v$, creating a new graph $G'$ by replacing $v$ with two non-adjacent vertices $v_1$ and $v_2$, ensuring that if $v_1$ and $v_2$ are subsequently merged (identified as a single vertex), the original graph $G$ is recovered. This process merges the outer face $f_{out}$ and inside face $f_{in}$ into a unified face. 
The constructed graph $G'$ is an outerplane graph, implying its weak dual is a forest. Critically, under the specific transformation applied, the weak dual structure of $G'$ remains identical to that of the original graph $G$. This equivalence establishes that the weak dual of $G$ itself must also form a forest,

\vspace{2mm}\textbf{Case 2:} $V(f_{out})\cap V(f_{in})=\emptyset$.\vspace{2mm}

If $G$ has no interconnected edge, then it is easy to check that the weak dual of $G$ is a star. So we assume that $G$ has an interconnected edge $uv$ such that $u\in V(f_{out})$ and $v\in V(f_{in})$. We perform an edge-based graph dissection along the edge $uv$, constructing a new graph $G'$ through vertex splitting operations. Specifically, vertex $u$ is divided into two non-adjacent vertices $u_1, u_2$, and vertex $v$ into two non-adjacent vertices $v_1, v_2$. In this construction, edges $u_1v_1$ and $u_2v_2$ are explicitly added to $G'$. This transformation preserves the original structure such that merging $u_1$ with $u_2$, and $v_1$ with $v_2$, in $G'$ perfectly reconstructs the original graph $G$. Again, we come to an outerplane graph $G'$ whose weak dual $F$ is a forest. Let $f_1$ and $f_2$ be the two faces incident with $uv$ in $G$. Since $f_1$ and $f_2$ correspond to two adjacent vertices $w_1$ and $w_2$ in the weak dual of $G$, now, the weak dual of $G$ is $F+w_1w_2$, which is a unicyclic graph.
\end{proof}

\begin{remark}
    The proof of Theorem \ref{unicyclic} also implies that if the weak dual of a 2-boundary plane graph is not a cycle, then it contains at least one leaf. The weak dual of an outerplane graph has numerous applications in the literature, including the subgraph isomorphism problem \cite{MR644795}, the induced matching problem \cite{MR2858009}, the outerplanar Tur\'an number problem \cite{MR4851767, MR4630877}, the edge-coloring problem \cite{li2024packingedgecoloringssubcubicouterplanar}, and the square coloring problem \cite{MR2379155, MR2291059}. We believe that the weak dual of a 2-boundary planar graph will also have similar applications in the future.
\end{remark}

A graph property is \textit{minor-closed} if it remains unchanged under vertex deletion, edge deletion, and edge contraction. In other words, if a graph has a minor-closed property $P$, then all of its minors will also exhibit property $P$. For instance, Kuratowski's Theorem states that a graph is planar if and only if it contains no minors isomorphic to $K_5$ or $K_{3,3}$, from which one can deduce that a graph is outerplanar if and only if it contains no minors isomorphic to $K_4$ or $K_{2,3}$. Hence, both planarity and outerplanarity are examples of minor-closed properties. For 2-boundary planarity, we propose the following question:

\begin{question}
Can 2-boundary planarity be characterized as a minor-closed property? Furthermore, if it is, does there exist a finite set $\mathcal{G}$ of graphs, such that a graph is 2-boundary planar if and only if it does not contain any minors isomorphic to any graph within $\mathcal{G}$?
\end{question}

It is important to note that for three classes of graphs, namely $\mathcal{G}_a,\mathcal{G}_b$, and $\mathcal{G}_c$, where $\mathcal{G}_a\subset \mathcal{G}_b \subset \mathcal{G}_c$, the fact that both $\mathcal{G}_a$ and $\mathcal{G}_c$ are minor-closed does not necessarily imply that $\mathcal{G}_b$ is minor-closed. For example, let $\mathcal{G}_a,\mathcal{G}_b$, and $\mathcal{G}_c$ be the class of outerplanar graphs, outer-1-planar graphs, and planar graphs. A graph is outer-1-planar if it can be drawn in the plane such that all vertices are in the outer face and each edge is crossed at most once. Outer-1-planar graphs were first introduced by Eggleton \cite{MR846198} under the notion of  \textit{outerplanar graphs with edge crossing number one} \cite{MR846198}. They were also known as \textit{pseudo-outerplanar graphs} \cite{MR2945171,MR3084275,MR3203677} in the literature. Hong et al.\,\cite{Hong2014ALA} showed that every outer-1-planar graph is planar. Nevertheless, outer-1-planar graphs are not closed under either edge contraction or subdivision. Consequently, it is impossible to characterize outer-1-planar graphs by prohibiting certain minors \cite{zbMATH06587983}. However, Auer et al.\,\cite{zbMATH06587983} revealed that if a graph is not outer-1-planar, then it must contain at least one of the six graphs in a specific set as a minor.

\begin{remark} \label{rmk}
    Now that we have mentioned outer-1-planar graphs, let us talk about the odd coloring of outer-1-planar graphs briefly. Li and Zhang \cite{dmtcs:8381} showed that every outer-$1$-planar graph $G$ contains 
\begin{enumerate}[label={\rm \textbf{(\alph*$\ref{rmk}$)}}]\setlength{\itemsep}{-3pt}
\item \label{a1} a vertex $v$ of degree $1$, or
\item \label{b1} a vertex $v$ with $N_G(v)=\{u,w\}$ such that $uw\in E(G)$, or
\item \label{c1} a vertex $v$ with $N_G(v)=\{u,w\}$ such that $uw\not\in E(G)$ and $G-v+uw$ is outer-1-planar, or
\item \label{d1} two adjacent triangles $xuvx$ and $yuvy$ with  $d_G(u)=d_G(v)=3$.
\end{enumerate}
With this in hand, we show that every outer-1-planar graph has an odd 5-coloring, confirming Conjecture \ref{conj} for this planar subclass. Suppose by contradiction that there is an odd non-5-colorable minimal outer-1-planar graph $G$, which contains at least one configuration among \ref{a1}, \ref{b1}, \ref{c1}, and \ref{d1}. However, \ref{a1}, \ref{b1}, \ref{c1}, and \ref{d1}
are respectively odd 5-reducible by Lemma \ref{cl-1}, \ref{cl-2}, \ref{cl-3}, and \ref{cl-4}, a contradiction.
\end{remark}

An additional interesting problem regarding 2-boundary planar graphs pertains to the testing algorithm, leading us to pose the following question:

\begin{question}
Can there be a linear-time algorithm that not only tests whether a given graph is $2$-boundary planar but also, if it exists, produces a $2$-boundary planar embedding in linear time?
\end{question}



\bibliographystyle{abbrv}
\bibliography{ref}







\end{document}